\documentclass[12pt]{article}
\usepackage[utf8]{inputenc}
\usepackage{color}
\usepackage{graphicx}
\usepackage{amssymb}
\usepackage{amsfonts,amsmath,amsthm}
\usepackage{float}
\usepackage{color}
\usepackage{setspace}
\usepackage{soul}
\usepackage[left=2 cm,right=2 cm,top=2cm,bottom=2cm]{geometry}
\date{}
\theoremstyle{definition}
\newtheorem{thm}{Theorem}[section]
\newtheorem{definition}{Definition}[section]

\newtheorem{cor}{Corollary}[section]
\newtheorem{example}[thm]{Example}

\theoremstyle{remark}
\newtheorem{remark}{\textbf{Remark}}[section]
\begin{document}
\Large
	\begin{center}
		Orbitwise expansive maps
		
		\hspace{10pt}

		Debasish Bhattacharjee$^1$,  Humayan Kobir$^2$
         , Santanu Acharjee$^{3}$\\
		
		$^{1,2,3}$ Department of Mathematics\\
        Gauhati University\\
        Guwahati-781014, Assam, India\\
		e-mails: $^{1}$debabh2@gmail.com, $^{2}$humayankobir60@gmail.com, $^{3}$sacharjee326@gmail.com
			\end{center}
	
	\hspace{10pt}
	\normalsize
\begin{abstract}
 This study defines an orbitwise expansive point (OE) as a point, such as $x$ in a metric space $(X,\rho)$, if there is a number $d>0$ such that the orbits of a few points inside an arbitrary open sphere will maintain a distance greater than $d$ from the corresponding points of the orbit of $x$ at least once. The point $x$ is referred to as the relatively orbitwise expansive point (ROE) in the previously described scenario if $d$ is replaced with the radius of the open sphere whose orbit is investigated and whose centre is $x$. 
 We also define OE (ROE) set. We prove that arbitrary union of OE (ROE) set is again OE (ROE) set and every limit point of an OE set is an OE point. We show that, rather than the other way around, Utz's expansive map or Kato's CW-expansive map implies OE (ROE) map. We utilise the concept of OE(ROE) to analyse a time-varying dynamical system and investigate its relevance to certain traits associated with expansiveness.\\
\end{abstract}	
\textbf{Key Points: Expansive map, CW$-$expansive map, Ball expanding map, Dense set, Limit point.} 

\section{Introduction} 
\hspace{0.5cm} Chaos is one of the consequences of non-linear dynamical systems, and it is distinguished by its sensitivity to the initial conditions. Let $f$ be a map on a metric space $(X,\rho)$.  An Orbit or a trajectory of a point $x$ is a sequence $\{f^n(x):n\in Z\}$. Since $1950$, when Utz \cite{Utz} introduced the unstable self-homeomorphism $f$ on a metric space $(X,\rho)$, the separation of trajectories has been a subject of interest to numerous researchers. This self-homeomorphism has the property that any two points, when progressing (forward or backward) through their trajectories, will find themselves separated at a distance greater than $d>0$ at some point. It is clear that expansiveness implies the notion of sensitivity to initial conditions, which is the kernel in the definition of Devaney’s chaos \cite{DL}. In $1952$, Schwartzman \cite{SZ} considered the expansiveness of $f$ in the positive time direction of trajectories. These functions are also known as positive expansive maps. Subsequently, the literature has incorporated the expansiveness in numerous variations. 
In $1970$, Reddy \cite{Reddy} replaced the uniform expansive constant by a pointwise expansive constant. According to Ruelle \cite{Ruelle}, a dynamical system $f :X\rightarrow X$ defined on a metric space $(X, \rho)$ is expanding if there exists $\epsilon>0$ and $\lambda>1$ such that
whenever $\rho(x, y) < \epsilon$ then, $\rho(f (x),f (y)) > \lambda \rho(x, y)$. The natural extension of it is that whenever $\rho(x, y) < \epsilon$ then, $\rho(f^n (x),f^n (y)) > \lambda \rho(x, y)$. In \cite{RD}, Thakkar and Das defined an expansive map on a time-varying dynamical system. A time-varying dynamical system can be seen in any moving image on a television screen \cite{RD}. In \cite{AF}, Fedeli defined a map that is positively ep-expansive on a dense subset of a metric space. In addition to the numerous variations of expansive maps, numerous researchers investigated expansiveness in a variety of spatial contexts. Otafudu et al. \cite{OO} conducted a study on expansive homeomorphisms in quasi-metric space, while M. Achigar \cite{AM} investigated expansive systems on lattices.  \\

In a chaotic regime, trajectory expansiveness should be visible but it is not restricted to every trajectory. In view of that, Kato \cite{HK}  said $f:X\rightarrow X$ is continuum-wise expansive if the  diameter of the set $\{f^n(x):x\in A\}$ is greater than some constant (called expansive constant), where $A$ is a sub continuum of $X$. Barwell et al. \cite{BC} defined an expanding map on a subset and an expanding ball on a compact metric space. In \cite{CM}, Morales defined $f:X\rightarrow X$ as $n-$expansive if for all $x\in X$, the closed ball $\{y\in X: \rho(f^n(x),f^n(y))\leq\delta, \forall n\in\mathbf{Z}\}$ contains at most $n$ elements, i.e., at most $n$ trajectories may not be visible, but the remaining will be visible. Later on, Li and Zhang \cite{JL} said that at most countable, trajectories are not visible, and the remaining will be visible. Since 2012, there has been a significant interest among academicians in $n-$expansive maps \cite{CM}. In \cite{1}, Lee et al. studied measure $N-$expansive system and in \cite{2}, Lee et al. studied kinematic $N-$expansive continuous dynamical system. In \cite{3}, Bomi Shin generalized the concept of expansive measure to continuum-wise expansive measure and continuum-wise $N-$expansive measure. In \cite{AB}, Barzanouni defined  $f:X\rightarrow X$ as finite orbit expansive if for every infinite set $A$, $f^n(A)\notin U_i$ for all $i=1,2,...,n$, where $\{U_1,U_2,...,U_n\}$ is a finite open cover of a compact metric space $(X,\rho)$, and if there is $A\in\{U_1,U_2,...,U_n\}$ such that for every $x\in X, \{y:\rho(f^n(x),f^n(y))\in A,\forall n\in\mathbf{Z}\}$ is finite then $f$ is said to be topologically finite expansive homeomorphism. So, our paper falls into this category. \\

Our initial proposal in this paper is to introduce orbitwise expansive (OE) maps, which contain expansive, CW-expansive maps. Subsequently, in the OE scenario, a relatively orbitwise expansive (ROE) map is established by substituting the constant expansion factor with an initial condition-dependent expansion factor. In our definitions, for each open sphere, there exists at least one point whose trajectory with the trajectory of the centre of that sphere must not be visible.  This paper is organized as follows: Section 2 is dedicated to the definition of important words and the statement of the theorems that will be utilised in subsequent sections. Section 3 introduces the definitions of OE map and OE point in a metric space. Additionally, we prove that while an expansive map and a CW-expansive map imply an OE map, the reverse is not often true. We establish generalisations of existing results and discover novel findings pertaining to dense sets and limit points. Section 4 introduces the concept of the ROE map in a metric space, presents several findings, and demonstrates that the arbitrary union of the ROE set is again ROE. In section 5, we  apply the concepts of OE and ROE in time-varying dynamical system and discuss some results and give one counter example that OE map may not be expansive map in  time-varying dynamical system. Section 6 contains a discussion.

\section{Preliminaries}
In this section, we recall some definitions and results that will be needed in the later sections.

\begin{definition}\cite{Utz}
Let $f$ be a homeomorphism of a metric space $(X,\rho)$ onto itself. Then, $f$ will be called unstable (expansive) on $X$, provided there is a number $\delta(f, X)>0$ (called an instability constant) such that corresponding to each pair of distinct points $x$, $y$ of $X$, there is an
integer $n(x, y)$ for which
$$\rho(f^n(x),f^n(y)) >\delta$$
\end{definition}
\begin{thm}\cite{Utz}
Suppose $f(X)=X$ is an unstable homeomorphism on a compact metric space $(X,\rho)$. Then, the function $f^m(X)=X$ is unstable for any integer $m\neq 0$.
\end{thm}
\begin{thm}\cite{BF}
Let $X$ and $Y$ be metric spaces with metrices $\rho_1$ and $\rho_2$, respectively, and $\phi$ is a self-expansive map on $X$. If $g$ is a homeomorphism of $X$ onto $Y$ such that $g^{-1}$ is uniformly continuous, then $g\phi g^{-1}$ is an expansive self-homeomorphism of $Y$.
\end{thm}
\begin{thm}\cite{BF}
If $A\subset X$, $X\setminus A$ is finite, and $\phi$ is a self-homeomorphism of $X$ that is expansive on $A$, then $\phi$ is expansive on $X$.
\end{thm}
\begin{definition}\cite{Reddy}
A homeomorphism $f$ on a compact metric $(X,\rho)$ is pointwise expansive if for each $x$ in $X$ there is a positive number $c(x)$ such that $x\neq y$ in $X$, we have $$\rho(f^n(x),f^n(y))>c(x), \text{ for some integer }n. $$
\end{definition}

\begin{definition}\cite{SM}
Let $(X,\rho)$ be a metric space. Then  $\phi$ is called uniformly expansive if there is a $\delta>0$ such that for each $\theta$ with $0<\theta\leq \delta$ there exists a positive integer
$k(\theta)$ such that if $x, y\in X$ satisfy $\rho(x, y)>\theta$ then $\rho(\phi^n(x),\phi^n(y)) >\delta$ for some
n with $\lvert n\rvert < k(\theta)$. The number $\delta$ is called a uniformly expansive constant for $\phi$.
\end{definition}
\begin{thm}\cite{SM}
Let $A$ be a dense subset of $X$ and suppose that $\phi$ is uniformly expansive on $A$. Then, $\phi$ is uniformly expansive on $X$.
\end{thm}
\begin{definition}\cite{HK}
A homeomorphism $f:X\rightarrow X$ is CW-expansive if there is a positive number $c>0$ (called expansive constant) such that if $A$ is a non-degenerate sub-continuum of $X$, then there is an integer $n=n(A)\in\mathbb{Z}$ such that $diamf^n(A)>c$, where $diamS=sup\{\rho(x,y):x, y\in S\}$ for any subset $S$ of $X$. 
\end{definition}
\begin{definition}\cite{HK}
    A map $f: X\rightarrow X$ is called positively continuum-wise expansive if there is a positive 
number $c > 0$ such that if A is a nondegenerate subcontinuum of X, then there is a natural 
number $n$ such that $diamf^n(A)>c$,  where $diamS=sup\{\rho(x,y):x, y\in S\}$ for any subset $S$ of $X$.
\end{definition}


\begin{definition}\cite{BC}
Let $(X, \rho)$ be a compact metric space, and $f : X\rightarrow X$ be continuous. If there are $\delta >0$, $\mu > 1$ such that $\rho(f(x), f(y)) \geq \mu \rho(x, y)$ provided that $x, y \in A\subseteq X$ and $\rho(x, y) < \delta$ then we say that $f$ is expanding on $A$. If there is a
neighbourhood $U$ of $A$ such that the property holds for every $x, y \in U$, we say that $f$ is expanding on a neighbourhood of $A$. In the case that $A = X$ we simply say that $f$ is expanding.
\end{definition}
\begin{definition}\cite{BC}
For a compact metric space $X$ with metric $\rho$ and a subset $A \subseteq X$,
we say that a continuous map $f : X \rightarrow X$ is ball expanding on $A$ if there are a $\mu > 1$ and a $\nu > 0$ such that for every $x \in A$ and every $\epsilon < \nu$ we have that
$B_{\mu\epsilon}(f(x)) \subseteq f (B_{\epsilon}(x))$.

\end{definition}

\begin{definition}\cite{CM}
Let $(X,\rho)$ be a metric space. A homeomorphism $f:X\rightarrow X$ is said to be $n-$expansive if there is a $n-$expansive constant $\delta>0$ for $f$ such that for every $x\in X$, $B_\delta(x)$ has at most $n$ elements, where $B_\delta(x)=\{y\in X:\rho(f^n(x),f^n(y))\leq\delta\}$.
\end{definition}

\begin{definition}\cite{JL}
    Let $(X,\rho)$ be a metric space. A homeomorphism $f:X\rightarrow X$ is said to be $\aleph_0-$expansive if there is a $\aleph_0-$expansive constant $\delta>0$ for $f$ such that for every $x\in X$, $B_\delta(x)$ has at most countable elements, where $B_\delta(x)=\{y\in X:\rho(f^n(x),f^n(y))\leq\delta\}$.
\end{definition}

\begin{definition}\cite{RD}
Let $(X, \rho)$ be a metric space and $f_n :X\rightarrow X$ be a sequence
of continuous maps, $n=0,1,2,...$. The time-varying map $F = \{f_n\}_{n=0}^\infty$ is said to be expansive if there exists a constant $c>0$ (called expansive constant) such that for any $x, y\in X, x\neq y, \rho(F_n(x), F_n(y))>c$, for some $n\geq 0$, where 
$F_n = f_n\circ f_{n-1}\circ...\circ f_1 \circ f_0$, for all $n = 0,1,2,...$.
\end{definition}
\begin{definition}\cite{RD}
Let $(X, \rho_1)$ and $(Y, \rho_2)$ be two metric spaces. Let $F=\{f_n\}_{n=0}^\infty$ and
$G = \{g_n\}_{n=0}^\infty$ be time-varying maps on $X$ and $Y$, respectively. If there is a homeomorphism $h: X\rightarrow Y$ such that $h\circ f_n=g_n \circ h$ for all $n=0, 1, 2,...$, then $F$ and $G$ are said to be conjugate (with respect to the map $h$) or $h-$conjugate. In particular, if $h: X \rightarrow Y$ is a uniform homeomorphism, then $F$ and $G$ are said to be uniformly conjugate or uniformly $h-$conjugate.
\end{definition}
\begin{thm}\cite{RD}
Let $(X, \rho_1)$ and $(Y, \rho_2)$ be metric spaces. Let $F=\{f_n\}^{\infty}_{n=0}$ and $G=\{g_n\}^{\infty}_{n=0}$ be time-varying maps on $X$ and $Y$, respectively, such that $F$ is uniformly conjugate to $G$. If $F$ is expansive on $X$, then $G$ is expansive on $Y$.
\end{thm}
\begin{thm}\cite{RD}
Let $(X, \rho)$ be a metric space, $F=\{f_n\}^{\infty}_{n=0}$ be a time-varying map
which is expansive on $X$ and $Y$ be an invariant subset of $X$, then restriction of $F$ to $Y$, defined by $F\lvert Y = \{ f_n\lvert Y \}$ is expansive.
\end{thm}
\begin{thm}\cite{RD}
Let $(X, \rho_1)$ and $(Y, \rho_2)$ be metric spaces, and $F=\{f_n\}^{\infty}_{n=0}$, $G=\{g_n\}^{\infty}_{n=0}$ be expansive time-varying maps on $X$ and $Y$ respectively. Then under metric $\rho$ on $X \times Y$ defined by
$\rho((x_1, y_1), (x_2, y_2)) = max\{\rho_1(x_1, x_2), \rho_2(y_1, y_2)\}$;
$(x_1, y_1),(x_2, y_2)\in X\times Y$,
time-varying map $F\times G =\{ f_n \times g_n\}^{\infty}_{n=0}$ is expansive on $X \times Y$. Hence, every finite
direct product of expansive, time-varying maps is expansive.
\end{thm}


\section{Orbitwise Expansive map}

For this section, we introduce orbitwise expansive map (OE) and study related properties. Throughout the paper, $S_\epsilon(x)$ is an open sphere centered at $x$ on a metric space $(X,\rho)$ 
and if $\rho$ is a usual metric, then $\rho(x,y)=\lvert x-y\rvert, \forall x,y\in X$.  Moreover, $n$ denotes a natural number throughout the paper. 

\begin{definition}
Let $(X,\rho)$ be a metric space and $f:X\rightarrow X$ be a map. A point $x\in X$ is said to be an OE point of $X$ with respect to $f$ if there exists a number $d(f,X)=d_X>0$ (called OE constant for $X$) such that for each $\epsilon>0$, we have
$f^n(S_{\epsilon}(x))\nsubseteq S_{d_X}(f^n(x))$, for some n.\\

If every point of $X$ is an OE point of $X$ with respect to $f$, then $f$ is a sensitive map \cite{JG} (or an OE map) on $X$ and $X$ is said to be an OE set with respect to $f$.
\end{definition}

\begin{definition}
Let $(X,\rho)$ be a metric space and $f:X\rightarrow X$ be a map. Then a point $x\in X$ is said to be an OE point of $A$ if there exists a number $d_A>0$ such that for each $\epsilon>0$, there exists at least one point $y(\neq x)\in A\cap S_\epsilon(x)$ such that $f^n(y)\notin S_{d_A}(f^n(x))$, for some $n$. In other words, $f^n(S_\epsilon(x)\cap A)\nsubseteq S_{d_A}(f^n(x))$, for some $n$.\\

A subset $A$ of $X$ is said to be an OE set if all the points of $A$ are OE points on $A$ with respect to $f$.
\end{definition}

\begin{remark}
If $x$ is an OE point of a subset $A$ of $X$ with respect to a function $f$, then it is an OE point of $X$ with respect to $f$.
\end{remark}
\begin{proof}
Since $x$ is an OE point of $A$ with respect to $f$, there exists a number $d_A>0$ such that for each $\epsilon>0$, there exists at least one point $y(\neq x)\in A\cap S_\epsilon(x)\subseteq X$ such that $f^n(y)\notin S_{d_A}(f^n(x))$, for some $n$. Hence, $x$ is an OE point of $X$ with respect to $f$.

\end{proof}
The converse of the above remark may not be true, we have the following example: 
 \begin{example}
 We consider the usual metric space $(\mathbb{R},\rho)$ and take a subset $A=[0,1]$ of $\mathbb{R}$. Let $f:\mathbb{R}\rightarrow\mathbb{R}$ be a function defined by $f(x)=2x$. Then, $x=2$ is an OE point of $\mathbb{R}$ but not on $A$. 
 \end{example}
\begin{remark}
If $f$ is an OE map on $(X,\rho)$ and $d_X$ is an OE constant, then for all $0<d<d_X$ is also an OE constant.
\end{remark}
\begin{proof}
    Since $f$ is an OE map on $(X,\rho)$. So, for each $x\in X$ there exists a number $d_X>0$ such that for each $\epsilon>0$, we have $f^n(S_{\epsilon}(x))\nsubseteq S_{d_X}(f^n(x)),\text{ for some } n$. Hence, for all $0<d<d_X$, we have,
    $f^n(S_{\epsilon}(x))\nsubseteq S_{d}(f^n(x)),\text{ for some } n$.
\end{proof}

\begin{thm}
Let $(X,\rho)$ be a metric space. Then,
 
$(1)$ if $f$ is a self-expansive map on $X$, then $f$ is an OE map on $X$.

$(2)$ if $f$ is a CW-expansive map on $X$, then $f$ is an OE map on $X$.
\end{thm}
\begin{proof}$(1)$
Let $f$ be an expansive map on $X$. Then there is an expansive constant $d_X>0$ such that for distinct $x,y\in X$ we have $\rho(f^n(x),f^n(y))>d_X, \text{ for some }n.$\\

Let $\epsilon>0$ be any real number, and $y(\neq x)\in S_\epsilon(x)$. Then, by definition of expansive map, we have $\rho(f^n(x),f^n(y))>d_X$, for some $n$. Therefore, there exists at least one point $y\in S_\epsilon(x)$ other than $x$ such that $\rho(f^n(x),f^n(y))>d_X, \text{ for some }n$. Thus, $ f^n(S_{\epsilon}(x))\nsubseteq S_{d_X} (f^n(x)), \text{ for some }n.$
Hence, $X$ is an OE set with respect to $f$ with OE constant $d_X$.\\

$(2)$ 
Let $f: X\rightarrow X$ be a CW-expansive map. Then, there exists a CW-expansive constant $c>0$ such that if $A$ is a non-degenerate subcontinuum of $X$, so we have $diam f^n(A)>c, \text { for some }n$.\\

Let $x\in X$. Then, for each $\epsilon>0$, there exists some non-degenerate subcontinuum $A$ of $X$ such that $A\subseteq S_\epsilon(x)$. So, $diam f^n(S_\epsilon(x))>c, \text{ for some }n$. It means there exists at least one point $y\in S_\epsilon(x)$ such that
$\rho(f^n(x),f^n(y))>c, \text{ for some }n$. It implies $f^n(S_\epsilon(x))\nsubseteq S_c{f^n(x)}, \text{ for some }n$. Thus, $x$ is an OE point of $X$ with respect to $f$. 

Hence, $f$ is an OE map on $X$ with OE constant $c$.

\end{proof}
\begin{remark}
OE maps generally differ from expansive and CW-expansive maps. We have the following example:
\end{remark}

\begin{example}
We consider the metric space $(\mathbb{R},\rho)$, where  $\rho$ is a usual metric. Let us define a function $f:\mathbb{R}\rightarrow \mathbb{R}$ as follows:
$$
f(x) = \begin{cases} 
\lambda & \text{if } x \in \mathbb{Q}, \\
\lambda x & \text{if } x \in \mathbb{Q}^\complement
\end{cases}, \quad \text{where } \lambda > 1 \text{ is a rational number}.
$$
Then, $f$ is an OE map on $\mathbb{R}$ but neither expansive nor CW-expansive on $\mathbb{R}$.

\end{example}

\begin{thm}
Let $A$ be a dense subset of a metric space $(X,\rho)$. If $f$ is an expansive map on $A$, then $f$ is an OE map on $X$.
\end{thm}
\begin{proof}
Let $f$ be an expansive map on $A$. Then, there is an expansive constant $d_A$ such that for distinct $x, y\in A$, we have $\rho(f^n(x),f^n(y))>d_A,\text{ 
for some }n$.\\

Let $x\in X\setminus A$. Then, we have to show that there exists $d_X>0$ such that for each $\epsilon>0$, we have $f^n(S_\epsilon(x)\cap A)\nsubseteq S_{d_X}(f^n(x))\text{, for some }n$.\\ 

We claim that there is atmost one point $y\in S_\epsilon(x)\cap A$ such that $d(f^n(x),f^n(y))<\frac{d_A}{2}$  for all $n$, i.e., $f^n(y)\in S_\frac{d_A}{2}(f^n(x)$ for all $n$. If possible, let $y$ and $z$ be two such points in $S_\epsilon(x)\cap A$. Then, for all $n$,
$\rho(f^n(x),f^n(y))<\frac{d_A}{2}$ and $\rho(f^n(x),f^n(z))<\frac{d_A}{2}$.\\

\hspace{0.4cm}Now, $\rho(f^n(y),f^n(z))\leq \rho(f^n(y),f^n(x))+\rho(f^n(x),f^n(z))\leq \frac{d_A}{2}+\frac{d_A}{2}=d_A \hspace{0.5cm}\forall n $
which contradicts the fact that $f$ is expansive on $A$. Thus, $x\in X\setminus A$ is an OE point of $A$. Hence $x$ is an OE point of $X$ with OE constant $d_X=\frac{d_A}{2}$ with respect to $f$. Therefore, $f$ is an OE map on $X$ with OE constant $d_X$.
\end{proof}

\begin{thm}
Let $(X,\rho)$ be a metric space, $x$ be a limit point of $A\subseteq X$, and $f$ be an OE map on $A$. Then, $x$ is an OE point of $A$.
\end{thm}
\begin{proof}
Let $f$ be an OE map on $A$ and $x$ be a limit point of $A$. Let $x\notin X\setminus A$. If possible, let $x$ be not OE point of $A$. Then, for each $d>0$, there exists $\epsilon>0$ such that $f^n(S_{\epsilon}(x)\cap A)\subseteq S_{d}(f^n(x))\text{, for all }n$.
Let, $d$ be fixed number such that $d\leq \frac{d_A}{2}$. Then there exists $\epsilon$ such that
$f^n(S_{\epsilon}(x)\cap A)\subseteq S_{d}(f^n(x))\text{, for all } n$.\\

Let $y\in S_\epsilon(x)\cap A$. We consider $\Delta<\epsilon-\rho(x,y)$. Then $S_\Delta(y)\cap A\subseteq S_\epsilon(x)\cap A$ implies  $f^n(S_\Delta(y)\cap A)\subseteq f^n(S_\epsilon(x)\cap A)\subseteq S_d(f^n(x)),\text{ for all }n$.\\

We claim that $f^n(S_\Delta(y)\cap A)\subseteq S_{d_A}(f^n(y))$. Let $z\in S_\Delta(y)\cap A$. Then $\rho(f^n(z),f^n(x))<d$ and $\rho(f^n(x),f^n(y))<d$, for all $n$. So, $\rho(f^n(y),f^n(z))<2d\leq d_A$, for all $n$. Hence $ f^n(S_\Delta(y)\subseteq S_{d_A}(f^n(y))$, for all $n$. Thus, $y$ is not an OE point of $A$ with respect to $f$, which is a contradiction. Hence, $x$ is an OE point of $A$ and so an OE point of $X$ with respect to $f$ with OE constant $d_X=\frac{d_A}{2}$.
\end{proof}
\begin{cor}
Let $A$ be a dense subset of a metric space $(X,\rho)$. If $f$ is an OE map on $A$, then $f$ is OE on $X$.
\end{cor}
\begin{proof}
Since $A$ is a dense subset of the metric space $(X,\rho)$, every point of $X$ is a limit point of $A$. So, every point of $X$ is an OE point of $A$ with respect to $f$ . Therefore, every point of $X$ is an OE point of $X$ with respect to $f$. Hence, $f$ is an OE map on $X$.
\end{proof}
\begin{cor}
If $A$ is an OE subset of a metric space $(X,\rho)$ with respect to a function $f$, then $\overline{A}$ is also an OE subset of $(X,\rho)$ with respect to the function $f$.
\end{cor}
\begin{proof}
Since every limit point of $A$ is an OE point of $A$ with respect to $f$. Therefore, $\overline{A}$ is an OE set with respect to $f$.
\end{proof}
\begin{thm}
Let $(X,\rho)$ be a metric space and $f$ be an OE map on $A$, where $A$ is any subset of $X$. If $x\notin A$ and $x$ is not a limit point of $A$, then it is not an OE point of $A$. 
\end{thm}
\begin{proof}
Since $x$ is not a limit point of $A$. So, there exists at least one $\epsilon>0$ such that $S_\epsilon(x)\cap A=\emptyset$. Therefore, for this $\epsilon$ there does not exist any point $y\in S_\epsilon(x)\cap A$, such that $f^n(S_\epsilon(x)\cap A)\nsubseteq S_{d_A}(f^n(x))$, for some $n$ holds, where $d_A$ is an OE constant for the set $A$. Hence, $x$ is not an OE point of $A$.
\end{proof}
\begin{thm}
Let $(X,\rho)$ be a metric space. Let $\{A_\lambda:\lambda\in\Delta\}$ be an arbitrary collection of subsets of $X$, and each $A_\lambda$ contains at least two points of $X$. Then,
    
$(i)$ $\cup_{\lambda\in\Delta} OE(A_\lambda)\subseteq OE(\cup_{\lambda\in\Delta} A_\lambda)$

$(ii)$ $OE(\cap_{\lambda\in\Delta} A_\lambda)\subseteq\cap_{\lambda\in\Delta} OE(A_\lambda)$\\
where $OE(A_\lambda)$ is the collection of all OE points of $A_\lambda$ with respect to a function $f$.

\end{thm}
\begin{proof}
$(ii)$
We first prove that $A\subseteq B$ implies $OE(A)\subseteq OE(B)$. Let $x\in OE(A)$. Then, there exists $d_A>0$ such that for each $\epsilon>0$ and for some $n$, we have
$\rho(f^n(x),f^n(y))>d_A , \text{ where } y\in S_\epsilon(x)\cap A$. Therefore, $\rho(f^n(x),f^n(y))>d_A$, where $y\in S_\epsilon(x)\cap B$ (as $A\subseteq B$). 
We consider $d=min\{d_A,d_B\}$ as the OE constant of $B$ and therefore $x\in OE(B)$. Thus, $OE(A)\subseteq OE(B)$.\\

 Clearly, $A_\lambda\subseteq \cup_{\lambda\in\Delta} A_\lambda, \text{for all } \lambda\in\Delta\implies OE(A_\lambda)\subseteq OE(\cup_{\lambda\in\Delta} A_\lambda), \text{for all } \lambda\in\Delta$. Hence, $\cup_{\lambda\in\Delta} OE(A_\lambda)\subseteq OE(\cup_{\lambda\in\Delta} A_\lambda) $. \\

$(ii)$ Let $x\in OE(\cap_{\lambda\in\Delta} A_\lambda)$. Then $x$ is an OE point of $\cap_{\lambda\in\Delta} A_\lambda$. Therefore, there exists $d_{\cap_{\lambda\in\Delta}A_\lambda}>0$ such that for each $\epsilon>0$, we have
$f^n(S_\epsilon(x)\cap (\cap_{\lambda\in\Delta} A_\lambda))\nsubseteq S_{d_{\cap_{\lambda\in\Delta}}A_\lambda}(f^n(x))\text{, for some }n$. Therefore, there exists at least one point $y\in S_\epsilon(x)\cap (\cap_{\lambda\in\Delta} A_\lambda)$ such that,
$\rho(f^n(x),f^n(y))>d_{\cap_{\lambda\in\Delta}A_\lambda}, \text{ for some }n$. So, $y\in\cap_{\lambda\in\Delta} A_\lambda \implies y\in A_\lambda$, for all $\lambda\in\Delta\implies\rho(f^n(x),f^n(y))>d_{\cap_{\lambda\in\Delta}A_\lambda}$, for some  $n$. 
We consider $d=min\{d_{A_\lambda},d_{\cap_{\lambda\in\Delta}A_\lambda}\}$ as the OE constant of $A_\lambda$ and so $x\in OE(A_\lambda)$, for all $\lambda\in\Delta$. Therefore, $x\in\cap_{\lambda\in\Delta} OE(A_\lambda)$. Hence, $OE(\cap_{\lambda\in\Delta} A_\lambda)\subseteq\cap_{\lambda\in\Delta} OE(A_\lambda).$

\end{proof}
\begin{cor}
Let $(X,\rho)$ be a metric space, and let $A$ be the collection of all OE points of $X$ with respect to $f$. Then $A$ is also an OE subset of $X$ with respect to $f$.
\end{cor}
\begin{remark}
In general, $OE(A\cap B)\neq OE(A)\cap OE(B)$, where $A$ and $B$ are two subsets of a metric space $(X,\rho)$ that contain at least two points of $X$. We have the following example:
\end{remark}
\begin{example}
We consider $f$ be an expansive map on $(\mathbb{R},d)$. For the subsets $A=\{\frac{1}{n}:n\in\mathbb{N}\},B=\{\frac{-1}{n}:n\in\mathbb{N}\}$, we have $A\cap B=\emptyset,\text{ and so } OE(A\cap B)=\emptyset$.
Clearly, $0\in OE(A)$ and  $0\in OE(B)$. Therefore, $OE(A)\cap OE(B)\neq \emptyset$. Thus, $OE(A\cap B)\neq OE(A)\cap OE(B)$.
\end{example}

\begin{thm}
    Let $(X,\rho)$ be a metric space. Let $\{A_i\}_{i=1}^{m}$ be any finite collection of subsets of $X$, and each $A_i$ contains at least two points of $X$. Then, $OE(\cup_{i=1}^{m}OE(A_i))=\cup_{i=1}^{m}OE(A_i)$, where $OE(A_i)$ is the collection of all OE points of $A_i$ with respect to a function $f$.
\end{thm}
\begin{proof}
   Let $x\in OE(\cup_{i=1}^{m}A_i)$. Then, $x$ is an OE point of $\cup_{i=1}^{m}A_i$. So, there exists a number $d_{\cup_{i=1}^{m}A_i}>0$ such that for each $\epsilon>0$, we have
   \begin{equation}\label{a}
f^n(S_\epsilon(x)\cap(\cup_{i=1}^{m}A_i))\nsubseteq S_{d_{\cup_{i=1}^{m}A_i}}(f^n(x)), \text{ for some }n.
   \end{equation}
   
   We claim that $x\in OE(A_i)$, for some $i\in\{1,2,...,m\}$. If possible let $x\notin OE(A_i)$, for all $i\in\{1,2,...,m\}$. So, for each $d>0$ and for each $i\in\{1,2,...,m\}$, there exists at least one $\epsilon_i>0$ such that 
   \begin{equation}\label{b}
    f^n(S_{\epsilon_i}(x)\cap A_i))\subseteq S_d(f^n(x)), \text{ for all }n.
\end{equation}
Let $\epsilon_0=min\{\epsilon_1,\epsilon_2,...,\epsilon_m\}$. Since equation (\ref{b}) holds for each $i\in\{1,2,...,m\}$. So, we have 
$$f^n(S_{\epsilon_0}(x)\cap(\cup_{i=1}^{m} A_i))\subseteq S_d(f^n(x)), \text{ for all }n.$$
which contradicts equation $(\ref{a})$.
Thus, $x\in OE(A_i)$, for some $i\in\{1,2,...,m\}$. Therefore, $x\in \cup_{i=1}^{m}OE(A_i)$. So, $OE(\cup_{i=1}^{m}A_i)\subseteq\cup_{i=1}^{m}OE(A_i)$. Also, by theorem $3.7$ $(i)$, we have $\cup_{i=1}^{m}OE(A_i)\subseteq OE(\cup_{i=1}^{m}A_i)$. Hence,  $OE(\cup_{i=1}^{m}OE(A_i))=\cup_{i=1}^{m}OE(A_i)$.
\end{proof}
The above theorem is not true for arbitrary union we have the following example:
\begin{example}
    We consider the usual metric space $(\mathbb{R},\rho)$ and $f$ be an OE map on $\mathbb{R}$. Suppose, $A_n=\{-\frac{1}{n},\frac{1}{n}:n\in\mathbb{N}\}$. Then, $0\in OE(\cup_{i=1}^{\infty}A_i)$ but $0\notin \cup_{i=1}^{\infty}OE(A_i)$. Thus, $OE(\cup_{i=1}^{\infty}A_i)\neq \cup_{i=1}^{\infty}OE(A_i)$.
\end{example}
\begin{thm}
Let $(X,\rho)$ be a metric space and $f:X\rightarrow X$ be a map.  Let $\{A_\lambda:\lambda\in\Delta\}$ be arbitrary collection of OE subsets of $X$  with respect to $f$. Then $\cup_{\lambda\in\Delta} A_\lambda$ is OE with respect to $f$. 
\end{thm}
\begin{proof}
Let $x\in\cup_{\lambda\in\Delta}A_\lambda$. Then, $x\in A_\lambda$, for some $\lambda$. Since, each $A_\lambda$ is an OE set, there exists $d_{A_\lambda}>0$ such that for each $\epsilon>0$, we have
$f^n(S_\epsilon(x)\cap A_\lambda)\nsubseteq S_{d_{A_{\lambda}}}(f^n(x)),\text{ for some }n$. Therefore, $f^n(S_\epsilon(x)\cap \cup_{\lambda\in\Delta}A_\lambda)\nsubseteq S_{d_{A_\lambda}}(f^n(x))$, for some $n$ (as $A_\lambda\subseteq\cup_{\lambda\in\Delta}A_\lambda)$.\\

Therefore, $x$ is an OE point of $\cup_{\lambda\in\Delta} A_\lambda$ with respect to $f$. Hence,
$\cup_{\lambda\in\Delta} A_\lambda$ is OE set with respect to $f$.


\end{proof}

\begin{cor}
Let $(X,\rho)$ be a metric space and $f$ be a map on $X$. Let $\{A_\lambda:\lambda\in\Delta\}$ be a class of subsets of $X$ such that $X=\cup_{\lambda\in\Delta} A_\lambda$ and each $A_\lambda$, $\lambda\in\Delta$ is OE with respect to $f$. Then $X$ is OE with respect to $f$.
\end{cor}  
\begin{thm}
Let $(X,\rho)$ be a compact metric space, and let $f(X)=X$ be a homeomorphism. Then, for any integer $m \neq 0$, $f^m(X)=X$ is an OE map on $X$ if and only if $f$ is an OE map on $X$.\\
\end{thm}
\begin{proof}
Let $f$ be an OE map on $X$. Then, for each $x\in X$, there exists $d_X>0$ such that for each  $\epsilon>0$, we have
\begin{equation}\label{c}
  f^n(S_\epsilon(x))\nsubseteq S_{d_X}(f^n(x)), \text{ for some  }n   
\end{equation}

Since $X$ is a compact metric space and $f(X)=X$ is a homeomorphism, then given $\xi>0,\exists \eta>0$ such that 
\begin{equation}\label{d}
    \rho(x,y)>\xi\implies \rho(f(x),f(y))>\eta
\end{equation}

Let $\phi(X)=f^m(X)$ and consider the homeomorphisms $f^i(X)$, $i=\pm1,\pm2,...,\pm m$. So by (\ref{d}), $\exists$ 
$ \eta_i>0$ $(i=\pm1,\pm2,...,\pm m)$ such that 
\begin{equation}\label{e}
\rho(x,y)>d\implies \rho(f^i(x),f^i(y))>\eta_i,\forall x,y\in X
\end{equation}

We show that $min\{\eta_i\}$ is an OE constant for $\phi(X)$. From $(\ref{c})$, we see that there exists at least one point $y\in S_\epsilon(x)$ such that,
 $\rho(f^n(x),f^n(y))>d_X$.
Now, one can find $r$ in such way that $0<\lvert rm-n\rvert\leq m$ and if $i$ is taken as $rm-n$, then
$f^i(f^n(s))=f^{rm}(s)=\phi^r(s),\forall s\in X$. So, due to equation $(\ref{e})$, $\rho(\phi^r(x),\phi^r(y))>min\{\eta_i\}=d{^/}_{X}(say)$.
It implies $\phi^r(S_\epsilon(x))\nsubseteq S_{d{^/}_{X}}(\phi^r(x))\text{ for some }r $. Hence, $f^m(X)=X$ is an OE map on $X$ with OE constant $d^{/}_{X}$.\\

Obviously, the converse is true.
\end{proof}

 

\begin{thm}
In a compact metric space, the OE map is independent of the choice of equivalent metrics if the function is continuous. 
\end{thm}

\begin{proof}
Let $\rho_1$ and $\rho_2$ be two equivalent metrics on a compact metric space $X$.  Let $(X,\rho_1)$ be an OE metric space with respect to $f$. We assume $x\in X$. Then, there exists $d_{\rho_1}>0$ such that for each $\epsilon>0$, we have
    $f^n(S_{\rho_1}(\epsilon,x))\nsubseteq S_{\rho_1}(d_{\rho_1},f^n(x)), \text{for some n}$
where $S_{\rho_1}(\epsilon,x)=S_\epsilon(x)$ with respect to the metric $\rho_1$. Since $\rho_1$, $\rho_2$ are equivalent metrics, for  $d_{\rho_1}>0$, $\exists$ $d_{\rho_2}>0$ such that $S_{\rho_2}(d_{\rho_2},t)\subset S_{\rho_1}(d_{\rho_1},t)$, for all $t\in X$.\\

Let us consider the open spheres $S_{\rho_2}(\epsilon,x)$ and $S_{\rho_1}(\epsilon,x)$ in such a way that $S_{\rho_2}(\epsilon,x)\subseteq S_{\rho_1}(\epsilon,x)$. Then, $$S_{\rho_2}(d_{\rho_2},f^n(x))\subseteq S_{\rho_1}(d_{\rho_1},f^n(x))\nsupseteq f^n(S_{\rho_1}(\epsilon,x))\supseteq f^n(S_{\rho_2}(\epsilon,x)),\text{ for some }n.$$ Hence, $(X,\rho_2)$ is an OE metric space with respect to the continuous map $f$ with the OE constant $d_{\rho_2}$.
\end{proof}
\begin{thm}
Let $(X,\rho_1)$ and $(Y,\rho_2)$ be two OE metric spaces with respect to the functions $f$ and $g$, respectively. Then $(X\times Y, \rho)$ is OE with respect to the function $f\times g$ under the metric $\rho$ on $X\times Y$ defined by 
$$\rho((x_1,y_1),(x_2,y_2))=max\{\gamma(\rho_1(x_1,x_2),\gamma(\rho_2(y_1,y_2)\}$$
where $\gamma$ is a non negative real valued bounded function.
\end{thm}
\begin{proof}

Since $(X,\rho_1)$ is OE with respect to $f$. So for $x\in X$, there exists $d_X>0$ such that for each $\epsilon>0$, there exists some  $n$ such that, $f^n(S_{\rho_1}(\epsilon,x))\nsubseteq S_{\rho_1}(d_X,f^n(x))$. It means that, there exists at least one point $x^{'}\in S_{\rho_1}(\epsilon, x)$ such that, $\rho_1(f^n(x),f^n(x^{'}))>d_X$, for some $n$.\\

Since $(Y,\rho_2)$ is OE with respect to $g$. So for $y\in Y$, there exists $d_Y>0$ such that for each $\epsilon>0$, there exists some  $n$ such that, $g^n(S_{\rho_2}(\epsilon,y))\nsubseteq S_{\rho_2}(d_Y,g^n(y))$. So, there exists at least one point $y^{'}\in S_{\rho_2}(\epsilon,x)$ such that, $\rho_2(g^n(y),g^n(y^{'}))>d_Y$, for some $n$.\\

     Now,
     $\rho((f\times g)^n(x,y),(f\times g)^n(x^{'},y^{'}))=\rho((f^n(x),g^n(y)),(f^n(x^{'}),g^n(y^{'})))$

     $\hspace{7cm}=max\{\gamma(\rho_1(f^n(x),f^n(x^{'})),\gamma(\rho_2(g^n(y),g^n(y^{'}))\}$

$\hspace{7cm}>max\{\gamma(d_X),\gamma(d_Y)\}=d_{X\times Y},\text{ for some }n$.\\

Therefore, $(f\times g)^n(x^{'},y^{'})\notin S_\rho(d_{X\times Y},(f\times g)^n(x,y)),\text{ for some }n.$
Hence, $(X\times Y,\rho)$ is OE with respect to the function $f\times g$ with OE constant $d_{X\times Y}$.   
\end{proof}
\begin{thm}
    Let $(X,\rho)$ be a metric space. Then, $X$ is OE set with respect to $f$ if every nondegenerate subcontinuum of $X$ is OE set with respect to $f$.
\end{thm}
\begin{proof}
Let $x\in X$. Then, $x\in A$ for some nondegenerate subcontinuum of $X$. Since each nondegenerate subcontinuum is OE set with respect to $f$. So, $x$ is an OE point of $A$ with respect to $f$. Thus, $x$ is an OE point of $X$ with respect to $f$. Hence, $X$ is an OE set with respect to $f$.
\end{proof}
The converse of the above theorem may not be true in general we have the following example:
\begin{example}
We consider the usual metric space $(\mathbb{R},\rho)$ and $A=\{n: n\in\mathbb{N}\}$ be any nondegenerate subset of $\mathbb{R}$. Suppose $f$ is a function on $\mathbb{R}$ defined by $f(x)=3x$. Then, $\mathbb{R}$ is an OE set with respect to $f$ but $A$ is not an OE set with respect to $f$.
\end{example}

\section{ Relatively Orbitwise Expansive map}
In this section, we defined relatively orbitwise expansive (ROE) map and proved some other results.
\begin{definition}
Let $(X,\rho)$ be a metric space and $f:X\rightarrow X$ be a map. A point $x\in X$ is said to be a ROE point with respect to $f$ if there exists $\epsilon_x>0$ such that for each $0<\epsilon<\epsilon_x$, we have $f^n(S_{\epsilon}(x))\nsubseteq S_{\epsilon}(f^n(x)),\text{ for some }n$. \\

If every point of $X$ is a ROE  point with respect to a function $f$, then $f$ is called ROE map on $X$. We also say that $X$ is  ROE set with respect to $f$.
\end{definition}
\begin{definition}
Let $(X,\rho)$ be a metric space and $f:X\rightarrow X$ be a map. Then a point $x\in X$ is said to be a ROE point of $A$ if there exists $\epsilon_x>0$ such that for each $0<\epsilon<\epsilon_x$, there exists at least one point $y(\neq x)\in S_\epsilon(x)\cap A$ such that $f^n(y)\notin S_\epsilon(f^n(x))$, for some $n$. In other words, $f^n(S_\epsilon(x)\cap A)\nsubseteq S_\epsilon(f^n(x))$, \text{ for some }n.\\

A subset $A$ of $X$ is said to be ROE set or $f$ is said to be ROE map on $A$ if all the points of $A$ are ROE points on $A$ with respect to $f$.
\end{definition}

\begin{remark}
If $x$ is a ROE point of a subset $A$ of $X$ with respect to $f$, then it is a ROE point of $X$ with respect to $f$.
    
\end{remark}
\begin{proof}
The proof is similar to remark $3.1$.
\end{proof}
\begin{remark}
 If $f$ is expansive map on $(\mathbb{R},\rho)$ and x$\in R$ is a ROE point, then there are uncountable number of points in  $S_\epsilon(x)$ such that $f^n(S_{\epsilon}(x))\nsubseteq S_{\epsilon}(f^n(x))$ .
\end{remark}
\begin{proof}
Since, $f$ is an expansive function, there exists an expansive constant $d_X>0$ such that for distinct  $x,y\in \mathbb{R}$, we have 
$\rho(f^n(x), f^n(y)) > d_X, \text{ for some }n$. As $x$ being a ROE point on $\mathbb{R}$, there exists $\epsilon_x>0$ such that for each $0<\epsilon<\epsilon_x$, we have $f^n(S_\epsilon(x))\nsubseteq S_\epsilon(f^n(x)), \text{ for some  }n$.\\
 
  Now, for each $y_i\neq x \in S_\epsilon(x) $, we have 
$\rho(f^n(x), f^n(y_i))>d_X\text{, for some }n$. Therefore, for each $0<\epsilon<d_X$ and for each $y_i\in S_\epsilon(x)$, we have $f^n(y_i)\notin  S_\epsilon (f^n(x)) $, for some $n$. It implies for all elements of $\{y_i\in S_\epsilon(x):i\in \Delta,0<\epsilon<d_X\}$, we have $f^n(y_i)\notin  S_\epsilon (f^n(x)) $, for some $n$. i.e., for all elements of $\{y_i\in S_\epsilon(x):i\in \Delta, 0<\epsilon<d_X\}$ the relation $f^n(S_{\epsilon}(x))\nsubseteq S_{\epsilon}(f^n(x))$ holds for some $n$.
\end{proof}
\begin{thm}
Let $(X,\rho)$ be a metric space. Then

$(1)$ if $f$ is a self-expansive map on $X$, then $f$ is a ROE map on $X$.

$(2)$ if $f$ is an OE on $X$, then $f$ is a ROE map on $X$. 
\end{thm}
\begin{proof}$(1)$ The proof is similar to theorem $3.1$. Hence, we skip it.\\ 






$(2)$ 
Let $f$ be an OE map on $X$ and let $x\in X$ be any point. Then there exists an OE constant $d_X>0$ such that for each $\epsilon>0$, we have
\begin{equation}\label{f}
f^n(S_\epsilon(x))\nsubseteq S_{d_X}(f^n(x)), \text{ for some }n 
\end{equation}

Let $\epsilon<d_X$. Then, from $(\ref{f})$, we have 
$f^n(S_\epsilon(x))\nsubseteq S_{\epsilon}(f^n(x)), \text{ for some }n $. 

Hence, $f$ is a ROE map on $X$.
\end{proof}
\begin{example}
Consider the usual metric space $(\mathbb{R},\rho)$. Let $f$ be a self-homeomorphism on $\mathbb{R}$ defined by $f^n(x)=2^{\frac{1}{n}}x$, 
 $x\in\mathbb{R}$. Then, $f$ is ROE but not OE on $\mathbb{R}$.
\end{example}

\begin{thm}
Let $(X,\rho)$ be a metric space, $x$ be a limit point of $A\subseteq X$, and $f$ be an expansive map on $A$. Then $x$ is a ROE point of $A$.
\end{thm}
 \begin{proof}

Let $x\in X\setminus A$. Since, $f$ is an expansive map on $A$, there exists an expansive constant $d_A>0$ such that for any two distinct point $x,y\in A$, we have 
 $\rho(f^n(x),f^n(y))>d_A, \text{for some }n$.\\

If possible, let $x$ not be a ROE point of $A$. Then, for any $0<\epsilon<d_A$ , we have 
\begin{equation}\label{g}
\rho(f^n(x),f^n(y))\leq\epsilon\hspace{0.2cm}  ,\forall y\in S_\epsilon(x)\cap A \text{ and }\forall n
\end{equation}

Since $x$ is a limit point of $A$, every open sphere contains an infinite number of points of $A$. Suppose $y,z\in S_\frac{d_A}{2}(x)\cap A$, taking $\epsilon=\frac{d_A}{2}$. So by $(\ref{g})$, we have 
$\rho(f^n(x),f^n(y))\leq\frac{d_A}{2}\hspace{0.2cm} \forall n$
 and $\hspace{0.2cm}\rho(f^n(x),f^n(z))\leq\frac{d_A}{2}\hspace{0.2cm} \forall n$.\\

Hence, $\rho(f^n(y),f^n(z))\leq \rho(f^n(y),f^n(x))+\rho(f^n(x),f^n(z))\leq d_A\hspace{0.2cm} \forall n$. But it contradicts the fact that $f$ is expansive on $A$. Therefore, $x$ is ROE point of $A$ with respect to $f$.
\end{proof}

\begin{thm}
Let $(X,\rho)$ be a metric space, $x$ be a limit point of $A\subseteq X$, and $f$ be a ROE map on $A$. Then $x$ is a ROE point of $A$.
\end{thm}
\begin{proof}
Let $x\in X\setminus A$. If possible, let $x$ not be the ROE point of $A$. Then, for  each $\epsilon_x>0$ and for each $0<\epsilon<\epsilon_x$ we have
 $f^n(S_{\epsilon}(x))\subseteq S_\epsilon(f^n(x))\text{, for all }n$.\\

 Since, $x$ is a limit point of $A$, for each open sphere, $S_{\epsilon}(x)$ contains infinite number of points of $A$, and therefore, for all $y\in (S_{\epsilon}(x)\cap A)$, we have 
 \begin{equation}\label{h}
     \rho(f^n(x),f^n(y))\leq \epsilon\text{, holds for all $\epsilon$ and } n
 \end{equation}
  
Now, we can construct an open sphere $S_{\epsilon}(y)$ in such a way that $(S_{\epsilon}(y)\cap A)\subset (S_{\epsilon}(x)\cap A)$ and contains at least two points and so for all $z\in S_{\epsilon}(y)\cap A$ and by $(\ref{h})$, we have $\rho(f^n(x),f^n(z))\leq \epsilon\text{, holds for all $\epsilon$ and }n$. Therefore, $f^n(S_\epsilon(y)\cap A)\subseteq S_\epsilon(f^n(y)), \forall \epsilon \text{ and } \forall n$. Thus, $y$ is not a ROE point of $A$ with respect to $f$ and so $A$ is not ROE set with respect to $f$, which is a contradiction. Hence, $x$ is a ROE point of $A$ with respect to $f$.
\end{proof}
\begin{cor}
Let $A$ be a dense subset of a metric space $(X,\rho)$. If $f$ is a ROE map on $A$, then $f$ is a ROE on $X$.
\end{cor}
\begin{proof}
    The proof is similar to corollary $3.1$.
\end{proof}
\begin{cor} Let $A$  be a ROE subset of a metric space $(X,\rho)$ with respect to a function $f$. Then, $\overline A$ is ROE set with respect to $f$.
\end{cor}
\begin{proof}
The proof is quite simple. We therefore ignore it.
\end{proof}
\begin{example} We consider the usual metric space $(\mathbb{R},\rho)$. Suppose $f$ is a ROE map on a subset $(a,b)$ or $[a,b)$ or  $(a,b]$ or $[a,b]$ of $\mathbb{R}$, where $a,b$ are distinct real numbers. Then the ROE points of these subsets with respect to the function $f$ are $[a,b]$. 
\end{example}

\begin{thm}
Let $(X,\rho)$ be a metric space. Let $\{A_\lambda:\lambda\in\Delta\}$ be an arbitrary collection of subsets of $X$, and each $A_\lambda$ contains at least two points of $X$. Then,
$$(i)\cup_{\lambda\in\Delta} ROE(A_\lambda)\subseteq ROE(\cup_{\lambda\in\Delta} A_\lambda)$$
$$(ii)ROE(\cap_{\lambda\in\Delta}A_\lambda)\subseteq\cap_{\lambda\in\Delta} ROE(A_\lambda)$$
where $ROE(A_\lambda)$ is the collection of all ROE points of $A_\lambda$ with respect to a function $f$.

\end{thm}
\begin{proof}
    The proof is similar to theorem $3.4$. So, we skip it.
\end{proof}
\begin{thm}
Let $(X,\rho)$ be a metric space, $f$ be a ROE map on $X$, and let $x$ be a limit point of $X$. Then for each $0<\epsilon<\epsilon_x$
, $f^n(x)\notin S_\epsilon(f^n(x))$ for an infinite number of points in $S_\epsilon(x)$.
\end{thm}
\begin{proof}
    Since $f$ is a ROE map on $X$, for each $x\in X$, there exists $\epsilon_x>0$ such that for each $0<\epsilon<\epsilon_x$, we have 
$f^n(S_\epsilon(x))\nsubseteq S_\epsilon(f^n(x)),\text{ for some }n$. It means, there exists at least one point $y\in S_\epsilon(x)$ such that $f^n(y)\notin S_\epsilon(f^n(x))$, for some $n$. Suppose, we have a collection of these types of points; we call it a solution set of $S_\epsilon(x)$. If possible, let the solution set of $S_\epsilon(x)$ be finite for each $\epsilon$.\\


Now, without loss of generality, we choose $\epsilon=\frac{1}{n}$. So the solution set of $S_1$ contains a finite number of points, say $\{x_1,x_2,...,x_n\}$, and we denote it by $B$. Next, the solution set of $S_\frac{1}{2}(x)$ contains less or equal points of $B$. In this way, if we reduce the radius of the open sphere, the solution set will contain fewer elements than before, and this procedure will stop after some finite steps. If we make the radius very small, it will contain only one point $x$ which contradicts the fact that $x$ is a limit point.
\end{proof}

\begin{thm}
Let $(X,\rho_1)$ and $(Y,\rho_2)$ be two metric spaces. If $g:X\rightarrow Y$ be a homeomorphism such that $g^{-1}$ is uniformly continuous, then $g\phi g^{-1}$ is a ROE of $Y$, where $\phi$ is a ROE map on $X$.
\end{thm}
\begin{proof}
Let $y\in Y$ and $0<\epsilon<\epsilon_y$, where $\epsilon_y$ is fixed. We have to show there exists at least one point $y_1\in S_\epsilon(y)$ such that $(g\phi g^{-1})^n(y_1)\notin S_\epsilon{(g\phi g^{-1})^n(y)}$. Since $g^{-1}$ is uniformly continuous, there exists $\epsilon>0$ such that $\rho_2(g(x_1),g(x_2))\leq\epsilon$ implies $\rho_1(x_1,x_2)<\delta$. Thus, $\rho_1(x_1,x_2)>\delta$ implies $\rho_2(g(x_1),g(x_2))>\epsilon$.\\

Let $y_1\in S_\epsilon(y)$. Then $g^{-1}(y_1)= x_1,g^{-1}(y)= x\in X$. So  $\rho_1(g^{-1}(y_1),g^{-1}(y))>\delta$ implies $\rho_2(g(x_1),g(x_2))>\epsilon$.
Since $\phi$ is ROE on $X$. So for $g^{-1}(y)=x$, there exists $\epsilon_x>0$ such that for each $0<\delta<\epsilon_x$, we have 
$\phi^n(S_\delta(x))\nsubseteq S_\delta(\phi^n(x))\text{, for some }n$. It means, there exists at least one point  $x_1\in S_\delta(x)$  such that, 
$\rho_1(\phi^n(x),\phi^n(x_1))>\delta, \text{ for some }n$. So,
$\rho_1(\phi^n(g^{-1}(y),\phi^n(g^{-1}(y_1))>\delta$
implies $\rho_1((\phi^n g^{-1})(y),(\phi^n g^{-1})(y_1))>\delta$. Hence, $ \rho_2(g(\phi^n g^{-1})(y),g(\phi^n g^{-1})(y_1))>\epsilon $ implies $ \rho_2((g\phi^n g^{-1})(y),(g\phi^n g^{-1})(y_1))>\epsilon$. So, $ (g\phi^n g^{-1})(y_1)\notin S_\epsilon{((g\phi^n g^{-1})(y))},\text{ for some }n$. Thus,\\
$(g\phi g^{-1})^n(y_1)\notin S_\epsilon{((g\phi g^{-1})^n(y))}$, for some $n$.
It implies $y$ is a ROE point on $Y$ with respect  $g\phi g^{-1}$. Therefore, $g\phi g^{-1}$ is a ROE on $Y$.
\end{proof}
\begin{thm}

Let $(X,\rho)$ be a metric space and $f:X\rightarrow X$ be a map.  Let $\{A_\lambda:\lambda\in\Delta\}$ be arbitrary collection of ROE subsets of $X$  with respect to $f$. Then $\cup_{\lambda\in\Delta} A_\lambda$ is ROE with respect to $f$. 
\end{thm}
\begin{proof}
Let $x\in\cup_{\lambda\in\Delta}A_\lambda$. Then, $x\in A_\lambda$, for some $\lambda$. As each $A_\lambda$ is ROE, there exists $\epsilon_x>0$ such that for each $0<\epsilon<\epsilon_x$, we have 
$f^n(S_\epsilon(x)\cap A_\lambda)\nsubseteq S_\epsilon(f^n(x)),\text{ for some }n$. It implies
$ f^n(S_\epsilon(x)\cap \cup_{\lambda\in\Delta}A_\lambda)\nsubseteq S_\epsilon(f^n(x))$, for some $n$.\\


Therefore, $x$ is a ROE point of $\cup_{\lambda\in\Delta} A_\lambda$ with respect to $f$. Hence,
$\cup_{\lambda\in\Delta} A_\lambda$  is ROE set with respect to $f$.
\end{proof}
\begin{cor}
Let $(X,\rho)$ be a metric space and $f:X\rightarrow X$ be a continuous map. Let $\{A_\lambda:\lambda\in\Delta\}$ be a class of subsets of $X$ such that $X=\cup_{\lambda\in\Delta} A_\lambda$ and each $A_\lambda$, $\lambda\in\Delta$ is ROE with respect to $f$. Then, $X$ is ROE with respect to $f$.
\end{cor}  

\section{OE and ROE of a time-varying map}
In this section, we discuss orbitwise expansive (OE) map and relative orbitwise expansive (ROE) map in a time-varying dynamical system.\\

We consider $(X, \rho)$ to be a metric space and $f_n:X\rightarrow X$ to be a sequence of  maps, $n = 0, 1, 2,...$. Let $F = \{f_n\}_{n=0}^\infty$ be a time-varying map ( or a sequence of maps) on $X$. We denote
$F_n = f_n\circ f_{n-1}\circ...\circ f_1 \circ f_0$, for all $n = 0,1,2,...$

\begin{definition}
The metric space $(X,\rho)$ is said to be OE with respect to the time-varying map $F=\{f_n\}_{n=0}^\infty$ if for $x\in X$, there exists  $d_X>0$  such that for each $\epsilon>0$, we have
$F_n(S_\epsilon(x))\nsubseteq S_{d_X}(F_n(x)), \text{ for some }n$.

\end{definition}

\begin{definition}
The metric space $(X,\rho)$ is said to be ROE with respect to the time-varying map $F=\{f_n\}_{n=0}^\infty$ if for $x\in X$, there exists $\epsilon_x>0$ such that for each $0<\epsilon<\epsilon_x$, we have 
$F_n(S_\epsilon(x))\nsubseteq S_{\epsilon}(F_n(x)), \text{ for some }n$.
\end{definition}

\begin{thm}
Let $(X, \rho)$ be a metric space and $f_n:X\rightarrow X$ be a sequence of
maps, $n = 0, 1, 2,...$. Let us assume that the time-varying map $F=\{f_n\}_{n=0}^\infty$ be expansive. Then, the time-varying map $F$ is OE as well as ROE on $X$.
\end{thm}

\begin{proof}
Since the time-varying map $F$ is expansive on $X$. So, there exists an expansive constant $d_X>0$ such that for every $x,y\in X$, we have
$\rho(F_n(x),F_n(y))>d_X\text{, for some }n$.\\

Let $\epsilon>0$ be any real number, and $y(\neq x)\in S_\epsilon(x)$. Then, by definition of expansive map, we have $\rho(F_n(x),F_n(y))>d_X$, for some $n$. Therefore, there exists at least one point $y\in S_\epsilon(x)$ other than $x$ such that $\rho(F_n(x),F_n(y))>d_X, \text{ for some }n$. Thus, $ F_n(S_{\epsilon}(x))\nsubseteq S_{d_X} (F_n(x)), \text{ for some }n.$
Hence, the time-varying map $F$ is an OE map on $X$. Hence, it is ROE on $X$.
\end{proof}
\begin{example}
We consider the metric space $(\mathbb{R},\rho)$ with the usual metric $\rho$. Let us define a function $f:\mathbb{R}\rightarrow \mathbb{R}$ as follows:

$$ f_n(x) = \begin{cases} 
n+1 & \text{if } x \in \mathbb{Q}, \\
(n+1) x & \text{if } x \in \mathbb{Q}^\complement
\end{cases}$$

Then, $F=\{f_n\}^{\infty}_{n=0}$ is an OE map on $\mathbb{R}$, but not expansive on $\mathbb{R}$.

\end{example}

\begin{proof}
Let $x\in \mathbb{R}$. Then, $x\in \mathbb{Q}$ or $\mathbb{Q}^c$. Without loss of generality, suppose $x\in \mathbb{Q}$. Then, for each $\epsilon>0$ the open sphere $S_\epsilon(x)$ contains infinitely many points $y\in \mathbb{Q}^c$ such that
 $\rho(F_n(x),F_n(y))=\lvert F_n(x)-F_n(y)\rvert=\lvert (n+1)-(n+1)!x \rvert=(n+1)\lvert n!x-1 \rvert>d$, for some n.
 Thus, $F_n(S_{\epsilon}(x))\nsubseteq
S_d(F_n(x))$, for some n. Hence, $F$ is an OE map on $A$.\\

Let $x,y\in \mathbb{Q}\subset \mathbb{R}$. Then, $\rho (F_n(x), F_n(y))=\lvert F_n(x)-F_n(y)\rvert=0$.
Hence, $F$ is not expansive on $\mathbb{R}$.

\end{proof}  
\begin{thm}
Let $(X,\rho_1)$ and $(Y,\rho_2)$ be two OE metric spaces with respect to the time-varying map $F=\{f_n\}_{n=0}^\infty$ and $G=\{g_n\}_{n=0}^\infty$ respectively. Then, $(X\times Y, \rho)$ is OE with respect to time-varying map $F\times G=\{f_n\times g_n\}_{n=0}^\infty$ under the metric $\rho$ on $X\times Y$ defined by 
$\rho((x_1,y_1),(x_2,y_2))=max\{\gamma(\rho_1(x_1,x_2),\gamma(\rho_2(y_1,y_2)\}$,
where $\gamma$ is a non-negative real valued bounded function.

\end{thm}

\begin{proof}
    For any $n\geq 0$,
    $$(F\times G)_n(x,y)=(F_n(x),G_n(y)),\text{ for }(x,y)\in X\times Y$$

Since $(X,\rho_1)$ is OE with respect to $F$. So, for $x\in X$, there exists $d_X>0$ such that for each $\epsilon>0$, there exists some integer $n$ such that,
    $F_n(S_{\rho_1}(\epsilon,x))\nsubseteq S_{\rho_1}(d_X,F_n(x))$. So, there exists at least one point $x^{'}\in S_{\rho_1}(\epsilon, x)$ such that 
     $$\rho_1(F_n(x),F_n(x^{'}))>d_X ,\text{ for some } n$$
     
Since $(Y,\rho_2)$ is OE with respect to $G$. So for $y\in Y$, there exists $d_Y>0$ such that for each $\epsilon>0$, there exists some integer $n$ such that 
    $G_n(S_{\rho_2}(\epsilon,y))\nsubseteq S_{\rho_2}(d_Y,G_n(y))$. i.e., there exists at least one point $y^{'}\in S_{\rho_2}(\epsilon,x)$ such that 
     $$\rho_2(G_n(y),G_n(y^{'}))>d_Y ,\text{ for some }n$$
     $$\text{ Now, }\rho((F\times G)_n(x,y),(F\times G)_n(x^{'},y^{'}))=\rho((F_n(x),G_n(y)),(F_n(x^{'}),G_n(y^{'}))$$
     $$\hspace{9.4cm}=max\{\gamma(\rho_1(F_n(x),F_n(x^{'})),\gamma(\rho_2(G_n(y),G_n(y^{'}))\}$$
$$\hspace{8.9cm}>max\{\gamma(d_X),\gamma(d_Y)\}=d_{X\times Y},\text{ for some n}$$

Thus, $ (F\times G)_n(x^{'},y^{'})\notin S_\rho(d_{X\times Y},(F\times G)_n(x,y))$ implies $(F\times G)_n(S_\rho(\epsilon,(x,y))\nsubseteq S_\rho(d_{X\times Y},(F\times G)_n(x,y)),\text{ for some }n$. Hence, $(X\times Y,\rho)$ is OE with respect to the time-varying map $F\times G=\{f_n\times g_n\}_{n=0}^\infty$ with OE constant $d_{X\times Y}$.
\end{proof}

\begin{thm}
Let $(X,\rho)$ be an ROE metric space with respect to the time-varying map $F=\{f_n\}_{n=0}^\infty$ and $Y$ be an invariant subset of $X$. Then restriction of $F$ to $Y$ is ROE with respect to $F\vline Y$.
\end{thm}
\begin{proof}
Let $y\in Y$. Then, $y\in X$. Since $(X,\rho)$ is ROE, there exists $\epsilon_y>0$ such that for each $0<\epsilon<\epsilon_y$, there exists some $n$ such that $F_n(S_\epsilon(y))\nsubseteq S_\epsilon(F_n(y))$.\\

Since $Y$ is an invariant subset of $X$. So for each $y\in Y$, we have $F_n(y)\in Y$. Hence, $F_n(S_\epsilon(y)\cap Y)\subset Y\text{ for each }\epsilon<\epsilon_y$. So, 
$(Y,\rho)$ is OE with respect to $F\vline Y$.
\end{proof}
\begin{thm}
Let $(X,\rho_1)$ and $(Y,\rho_2)$ be two metric spaces. If $g$ is a homeomorphism of $X$ onto $Y$ such that $g^{-1}$ is uniformly continuous, then $gFg^{-1}$ is a ROE of $Y$, where $F=\{f_n\}^{\infty}_{n}$ is a ROE map on $X$.
\end{thm}
\begin{proof}
Let $y\in Y$ and $0<\epsilon<\epsilon_y$, where $\epsilon_y$ is fixed. We have to show there exists at least one point $y_1\in S_\epsilon(y)$ such that $(gFg^{-1})^n(y_1)\notin S_\epsilon{(gFg^{-1})^n(y)}$
$\text{i.e., we have to show }(gF_ng^{-1})(y_1)\notin S_\epsilon{(g\phi^n g^{-1})(y)}$, for some $n$.\\

Since $g^{-1}$ is uniformly continuous, there exists $\epsilon>0$ such that $\rho_2(g(x_1),g(x_2))\leq\epsilon$ implies $\rho_1(x_1,x_2)<\delta$.
$\text{Thus, } \rho_1(x_1,x_2)>\delta$ implies $\rho_2(g(x_1),g(x_2))>\epsilon$.\\

Let $y_1\in S_\epsilon(y)$. Then, $g^{-1}(y_1)= x_1\text{ and }g^{-1}(y)= x\in X$. So,  $\rho_1(g^{-1}(y_1),g^{-1}(y))>\delta$ implies $\rho_2(g(x_1),g(x_2))>\epsilon$.\\

Given, $F$ is ROE on $X$. So, for $g^{-1}(y)=x$, there exists $\epsilon_x>0$ such that for each $0<\delta<\epsilon_x$, we have 
$F_n(S_\delta(x))\nsubseteq S_\delta(F_n(x))\text{, for some }n$. $\text{i.e., there exists at least one point } x_1\in S_\delta(x) \text{ such that, }$ 
$\rho_1(F_n(x),F_n(x_1))>\delta, \text{ for some }n$.\\

Now, $\rho_1(F_n(g^{-1}(y),F_n(g^{-1}(y_1))>\delta$ implies
 $\rho_1((F_n g^{-1})(y),(F_n g^{-1})(y_1))>\delta,\text{ for some }n$. Hence, $ 
\rho_2(g(F_n g^{-1})(y),g(F_n g^{-1})(y_1))>\epsilon$ implies $\rho_2((gF_n g^{-1})(y),(gF_n g^{-1})(y_1))>\epsilon,\text{ for some }n$. Therefore, $ gF_n g^{-1}(y_1)\notin S_\epsilon{(gF_n g^{-1})(y)},\text{ for some }n$
implies $y$ is a ROE point on $Y$ with respect  $gFg^{-1}$.
Hence, $gFg^{-1}\text{ is a ROE on }Y$.
\end{proof}

\begin{thm}
Let $(X,\rho_1)$ and $(Y,\rho_2)$ be two metric spaces. Let $F=\{f_n\}_{n=0}^\infty$ and $G=\{g_n\}_{n=0}^\infty$ be time-varying maps on $X$ and $Y$, respectively, such that $F$ is uniformly conjugate to $G$. If $(X,\rho_1)$ is ROE with respect to $F$, then $(Y,\rho_2)$ is ROE with respect to $G$.
\end{thm}
\begin{proof}
Since $F$ is uniformly conjugate to $G$, thus there exists a uniform homeomorphism $h:X\rightarrow Y$ such that $h\circ f_n=g_n\circ h$ for all $n\geq 0$ which implies $h\circ F_n=G_n\circ h$ for all $n\geq 0$.\\

Since $h$ is a uniform homeomorphism, $h^{-1}$ is uniformly continuous, there exists $\epsilon>0$ such that, 
$\rho_2(h(x_1),h(x_2))\leq\epsilon\text{ implies }\rho_1((x_1),(x_2))<d_X$. $\text{ Thus, }\rho_1(x_1,x_2)>d_X$ implies $\rho_2(h(x),h(y))>\epsilon$.\\

Let $y_1\in S_\epsilon(y)$ and $0<\epsilon<\epsilon_y$, where $\epsilon_y$ is fixed. Then $h(x_1)=y_1$ and $h(x)=y \in Y$. So, 
$\rho_1(x,x_1)>d_X$ implies $\rho_2(h(x),h(x_1))>\epsilon$.
As $F$ is OE on $X$, for $x\in X$, there exists $d_X>0$ such that for each $\epsilon>0$, we have  $F_n(S_{\rho_1}(\epsilon,x))\nsubseteq S_{\rho_1}(d_X,F_n(x)),\text{ for some }n$. i.e., there exists at least one point $x_1\in S_{\rho_1}(\epsilon,x)$ such that $\rho_1(F_n(x),F_n(x_1))>d_X, \text{ for some }n$. It implies $\rho_2(h(F_n(x),h(F_n(x_1))>\epsilon, \text{ for some }n$. We can write it as $\rho_2(G_n(h(x),G_n(h(x_1))>\epsilon, \text{ for some }n$. So, it implies $\rho_2(G_n(y),G_n(y_1))>\epsilon, \text{ for some }n$. Therefore, $G_n(S_\epsilon(y))\nsubseteq S_\epsilon(G_n(y)), \text{ for some }n$.\\

Hence, $(Y,\rho_2)$ is ROE with respect to $G$.
\end{proof}
 \section{Discussion}
Utz \cite{Utz} introduced an expansive map that, after a certain number of iterations, separates every unique point on a metric space by a uniform constant. Kato \cite{HK}, on the other hand, developed a CW-expansive map for each non-degenerate subcontinuum on a metric space, which is a more inclusive definition of an expansive map. This study aims to establish a definition for the expansiveness of the orbit of a point on a metric space in relation to a function. We refer to such points as OE (ROE) points. Furthermore, if all points in a set are OE (ROE) points, we classify it as an OE (ROE) set. Our notion encompasses a broader category of chaotic maps than the one encompassed by expansive maps. We demonstrate numerous properties regarding the OE(ROE) maps using a novel approach, resulting in a plethora of new results and instances. An expansive map is defined by R. Das \cite{RD} in a dynamical system that varies over time. We establish the OE(ROE) notions in time-varying dynamical system and provide several outcomes of an expansive map in the form of OE(ROE) notion.
 From this paper, we conclude the below relationship:
\begin{equation}\label{i}
    \text{Expansive map}\implies \text{CW-expansive map}\implies\text{ OE map}\implies \text{ ROE map}
\end{equation}

 The following question remain open:\\

 $(1)$ Does the relation $(9)$ become equivalent under some conditions?\\

 $(2)$  Does an OE(ROE) homeomorphism  exclusively follow the flow $(9)$, not vice-versa?

\end{document}